\documentclass[twoside,a4paper]{article}

\usepackage{amsmath,graphicx,amssymb,fancyhdr,amsthm,enumerate,textcomp,amscd} 
\usepackage[usenames]{color}
\usepackage[all]{xy}
\newtheorem{thm}{Theorem}[section] 
\newtheorem{cor}[thm]{Corollary} 
\newtheorem{lem}[thm]{Lemma} 
\newtheorem{prop}[thm]{Proposition}
 
\theoremstyle{definition} 
\newtheorem{defn}[thm]{Definition}
  
\theoremstyle{remark}  
\newtheorem{rem}[thm]{Remark}  
\def\beq{\begin{eqnarray}}  
\def\eeq{\end{eqnarray}}  
\def\bsp{\begin{split}}  
\def\esp{\end{split}}

\def\d{\mathrm{d}}  
\def\diag{\mathrm{diag}}  
 
\def\J{ J } 

\def\R{ \mathbb{R}}
  
\def\A{{\sf A}}  
\def\C{{\mathbb{C}}}

\newcommand{\mb}[1]{{\mathbb #1}}   
   
\newcommand{\mbold}[1]{\mbox{\boldmath{\ensuremath{#1}}}}

\begin{document}   
   
\title{\Large\textbf{Wick rotations and real GIT}}  
\author{{\large\textbf{Christer Helleland} and \textbf{Sigbj\o rn Hervik }    }
 \vspace{0.3cm} \\     
Faculty of Science and Technology,\\     
 University of Stavanger,\\  N-4036 Stavanger, Norway         
\vspace{0.3cm} \\      
\texttt{ christer.helleland@uis.no}, \\
\texttt{sigbjorn.hervik@uis.no} }     
\date{\today}     
\maketitle   
\pagestyle{fancy}   
\fancyhead{} 
\fancyhead[EC]{Helleland  and  Hervik}   
\fancyhead[EL,OR]{\thepage}   
\fancyhead[OC]{Wick rotations and real GIT}   
\fancyfoot{} 

\begin{abstract}
We define Wick-rotations by considering pseudo-Riemannian manifolds as real slices of a holomorphic Riemannian manifold. From a frame bundle viewpoint Wick-rotations between different pseudo-Riemannian spaces can then be studied through their structure groups which are real forms of the corresponding complexified Lie group (different real forms $O(p,q)$ of the complex Lie group $O(n,\C)$). In this way, we can use real GIT (geometric invariant theory) to derive several new results regarding the existence, and non-existence, of such Wick-rotations.  
As an explicit example, we Wick rotate a known $G_2$-holonomy manifold to a pseudo-Riemannian manifold with split-$G_2$ holonomy. 

\end{abstract}

\section{Introduction}
In this paper we will study so-called Wick-rotations which was first used in physics as a mathematical trick relating Minkowski space to flat Euclidean space. Here we will generalise the concept of Wick-rotations to more general pseudo-Riemannian geometries by considering the complexification to a holomorphic Riemannian manifold. The real pseudo-Riemannian manifolds will now manifest themselves as real slices of the complex holomorphic geometry. Utilizing this description we define Wick-rotations of pseudo-Riemannian manifolds as well as the stronger concept of a \emph{standard} Wick rotation. 

There are previous works considering complex geometry and Wick-rotations in different contexts \cite{complexGR,exact,OP,Loz,Visser,PV}. Here, we will define the Wick-rotations based on observations made in \cite{OP} which is related to the definition of \emph{Wick-related spaces} given in \cite{PV}. In fact, we adopt this definition but define the stronger concepts of \emph{Wick-rotations} and \emph{standard Wick-rotations}. This enables us to connect  the study of Wick-rotations to real GIT \cite{BC,RS} which recently has seen its appearence in the classification of pseudo-Riemannian geometries \cite{align,minimal}. Using old, as well as some new, results from real GIT we  give several results regarding the possibily of Wick-rotating pseudo-Riemannian spaces to different signatures (see also \cite{HH}). 

In this paper we will reserve the notion of \emph{Riemannian space} to the case when the metric is positive definite (of signature $(++..+)$). The Lorentzian case is the case of signature $(-++..+)$. Note also that the existence of the "anti-isometry" which switches the sign of the metric: $g\mapsto -g$. This anti-isometry induces the group isomorphism $O(p,q)\rightarrow O(q,p)$ and hence our results are independent under this map.

\section{Mathematical Preliminaries}
Let $E$ be a complex vector space. By considering only scalar multiplications by $\R\subset\C$, we can define a real vector space $E_\R$ whose points are identical to $E$. Multiplication with $i$ in $E$ defines an automorphism $\J:E_\R\longrightarrow E_\R$ satisfying:
\beq \label{J^2}\J\circ\J=-{\rm Id}. \eeq 
An endomorphism $\J$ satisfying eq.(\ref{J^2}) is called a \emph{complex structure} of $E_\R$. This correspondence between the complex vector space $E$ and the real vector space $E_\R$ equipped with  the  complex structure $\J$ defines an isomorphism of these categories. 

If we consider a complex vector subspace $W\subset E$, then its corresponding real vector space $W_\R$ is a vector subspace of $E_\R$ being invariant under $\J$. On the other hand, if $W$ is a vector subspace of $E_\R$ being invariant under $\J$ then we say that $W$ is a \emph{complex subspace}. 

The linear map $\J$ can also be extended to the complexification $E^{\C}:=E_\R\otimes_\R\C$. The complexification $E^{\C}$ can then be split into the sum of the eigenspaces: 
\[ E^{1,0}=\{u\in E_\R\otimes_\R\C\big{|}\J u=iu\} \] 
and 
\[ E^{0,1}=\{u\in E_\R\otimes_\R\C\big{|}\J u=-iu\}. \] 

\subsection{Real form of a complex vector space}
A complex vector space $E$ may be the complexification of a real vector space $W$: 
\[ E=W^{\C}:=W\otimes_\R\C,\] 
in which case we will call $W$ a \emph{real form} of $E$. 

Assume now that $W$ is a vector space over $\R$. Then $W$ is naturally a real form of the complexification: $W^{\C}$, indeed the field extension $\R\hookrightarrow \C$ induces the inclusion $W\hookrightarrow W^{\C}, \ \ w\mapsto w\otimes 1$. Furthermore, complex conjugation in $\C$ gives rise to an anti-linear involution $\rho$ of $W^{\C}$:
\[ \rho(w\otimes z)=w\otimes \bar{z}.\] 
The fixed-point set of $\rho$ is $W$. Such a map is called a \textsl{conjugation map} of $W^{\C}$ associated to $W$.

\

Some special examples of conjugation maps can be easily found among semi-simple complex Lie algebras. For example $\mathfrak{sl}(2,\C)$ has real forms $\mathfrak{su}(2)$ and $\mathfrak{sl}(2,\mathbb{R})$ associated to the conjugation maps $X\mapsto -X^{\dagger}$ and $X\mapsto \bar{X}$ respectively.

\begin{defn}
Let $E$ be a complex vector space with a complex structure $\J: E_\R\longrightarrow E_\R$. Then real linear subspace $W\subset E_\R$ is called totally real if $W\cap \J(W)=0$. \end{defn} 
In particular we see that if $W$ is a maximal totally real subspace, then $W$ is a real form of $E$. We note that if $W$ is totally real implies that the composition $$W^{\C}:=W\otimes_{\R} \C \longrightarrow E^{\C} \longrightarrow E^{1,0} \cong E,$$ where the second map is the projection onto $E^{1,0}$, is injective. This is important to us in the following as this implies that there is a connection between the complexification of real pseudo-Riemannian geometry and holomorphic Riemannian geometry. 
We note also that if $W$ is a real form of $E$ then the complex dimension of $E$ and the real dimension of $W$ are equal. 

\subsection{Real slices} 

\begin{defn}
A \emph{holomorphic inner product} is a complex vector space $E$ equipped with a non-degenerate complex bilinear form $g$. 
\end{defn}
For a holomorphic inner product space $E$ we can always choose an orthonormal basis. By doing so we can identify $E$ with $\C^n$ and the holomorphic inner product can be written as
\beq \label{g0}
g_0(X,Y)=X_1Y_1+...+X_nY_n, 
\eeq
where $X=(X_1,...,X_n)$ and $Y=(Y_1,...,Y_n)$. 

Using this orthonormal basis it is also convenient to consider the group of transformations leaving the holomorphic inner product invariant. Consider a complex-linear map $A: E\longrightarrow E$. Using an orthonormal basis, we can represent the map by a complex matrix $ \A: \C^n \longrightarrow \C^n$. Requiring that $g_0(A(X),A(Y))=g_0(X,Y)$, for all $X,Y$, implies that $\A^t\A={\sf 1}$. Consequently, the matrix $\A$ must be  a complex orthogonal matrix; i.e., $\A\in O(n,\C)$. 

\begin{defn}\label{realslice}
Given a holomorphic inner product space $(E,g)$. Then if $W\subset E$ is a real linear subspace for which $g\big{|}_W$ is non-degenerate and real valued, i.e., $g(X,Y)\in \R,~ \forall X,Y\in W$, we will call $W$ a \emph{real slice}. 
\end{defn}
Some standard examples of real slices can be found by considering the holomorphic inner product space $(\C^n,g_0)$ with standard basis $\{e_1,...,e_n\}$. The real subspace: 
\beq
W=\R^n_p:=\text{span}\{ie_1,...,ie_p, e_{p+1},..., e_{n}\},
\eeq
is a real slice for any $0\leq p \leq n$. The restriction of $g_0$ to $W$ in this case is the standard pseudo-Euclidean metric with signature $(p,n-p)$. Using the standard coordinates $z_k=x_k+iy_k$ for $\C^n$, we see that the restriction $x_1=...=x_p=y_{p+1}=...=y_n=0$ gives us the real slice $\R^n_p$. 

Let us assume that $W$ and $\widetilde{W}$ are real slices of $(\C^n,g_0)$. Consider the real slice $W$ with real non-degenerate bilinear form $h$. By choosing a pseudo-orthonormal basis, we can write:
\beq \label{h}
h(X,Y)=-X_1Y_1-...-X_pY_p+X_{p+1}Y_{p+1}+...+X_nY_n, 
\eeq
where $X=(X_1,...,X_n)$ and $Y=(Y_1,...,Y_n)$ (real) for some $p$. This space has complexification $W^{\C}$ by allowing $X$ and $Y$ to be complex numbers, by: $X_k\mapsto z_k=x_k+iX_k$, $1\leq k\leq p$ and $X_k\mapsto z_k=X_k+iy_k$, $p+1\leq k\leq n$ (similarly for $Y$). By restricting to $W$, we see that $W^{\C}$ is $\C^n$ in an orthonomal frame. 
Doing the same for $\widetilde{W}$ we note that $(\widetilde{W})^{\C}$ is also $\C^n$ in (possibly another) orthonormal frame. However, since orthornormal frames are related by the action of the group $O(n,\C)$, the real slices $W$ and $\widetilde{W}$ are related via the action of the group $O(n,\C)$ on $\C^n$. Indeed, since any $n$-dimensional complex holomorphic inner product space $(E,g)$ can be identified with $(\C^n,g_0)$, any two real slices of $E$ are related through the action of $O(n,\C)$ on $E$. 

\begin{defn}\label{CartanInv} Let $W\subset (E,g)$ be a real slice. We say an involution $W\xrightarrow{\theta} W$, is a \emph{Cartan involution} of $W$, if $g_{\theta}(\cdot,\cdot):=g\big{|}_{W}(\cdot,\theta(\cdot))$, is an inner product on $W$. \end{defn}

Of course Cartan involutions always exist as linear maps, and the definition generalises the notion of a Cartan involution of a semi-simple Lie algebra. Indeed special examples can be found within semi-simple real forms $\mathfrak{g}$ of a complex semi-simple Lie algebra $\mathfrak{g}^{\C}$, which are real slices w.r.t the holomorphic Killing form: $-\kappa(-,-)$, on the complex Lie algebra $\mathfrak{g}^{\C}$. So a Cartan involution: $\mathfrak{g}\xrightarrow{\theta} \mathfrak{g}$, will in this case in fact be an involution of Lie algebras, and will be unique up to conjugation by inner automorphisms of $\mathfrak{g}$. Explicit examples of real slices are the pseudo-orthogonal real Lie algebras $\mathfrak{o}(p,q)$ of $\mathfrak{o}(n,\mathbb{C})$ with signatures $\Big{(}\binom{p}{2}+\binom{q}{2}, 2pq\Big{)}$. 

\

It is also important to note the following (see e.g., \cite{PV}):
\begin{prop} 
The real slices of a holomorphic inner product space are totally real subspaces. 
\end{prop}

\subsection{Compatible real forms}\label{compatible}

Associated to any real form $W$ of a complex vector space $E\cong W^{\C}$ we know that there is a conjugation map $E\xrightarrow{\sigma} E$ with fix points $W$. The space $E$ may have another real form $\widetilde{W}$, also with a conjugation map $\tilde{\sigma}$ which fixes pointwise $\widetilde{W}$. So we have the notion of \textsl{compatibility} among two real forms in the following definition:

\begin{defn} The two real forms $W$ and $\widetilde{W}$ of $E$ are said to be compatible if their conjugation maps commute, i.e $[\sigma,\tilde{\sigma}]=0$.  \end{defn}

For two compatible real forms $W$ and $\widetilde{W}$ of $E$ we may write: $$W=(W\cap\widetilde{W})\oplus (W\cap i\widetilde{W}) \ \ and \ \ \widetilde{W}=(W\cap\widetilde{W})\oplus (\widetilde{W}\cap iW).$$

In the case of Lie algebras the real forms will have conjugation maps which are also real Lie homomorphisms. As an example consider the real forms $\mathfrak{o}(p,q)$ and $\mathfrak{o}(\tilde{p},\tilde{q})$ embedded into $\mathfrak{o}(n,\mathbb{C})$ with $n=p+q=\tilde{p}+\tilde{q}$, with corresponding conjugation maps: $$X\mapsto -I_{p,q}\bar{X}I_{p,q}, \ \ \ \ X\mapsto -I_{\tilde{p},\tilde{q}}\bar{X}I_{\tilde{p},\tilde{q}}.$$ It is easy to see that $\mathfrak{o}(p,q)$ is compatible with $\mathfrak{o}(\tilde{p},\tilde{q})$, and also observe that a Cartan involution for both real forms may be chosen to be $X\mapsto \bar{X}$, i.e the Cartan involutions also commute, and we may choose the compact real form: $$\mathfrak{o}(n)=\{X\in \mathfrak{o}(n,\mathbb{C})|X=\bar{X}\},$$ which will be compatible with both $\mathfrak{o}(p,q)$ and $\mathfrak{o}(\tilde{p},\tilde{q})$. 
This means that if $$\mathfrak{o}(p,q)=\mathfrak{t}\oplus \mathfrak{p}, \ \ \ \mathfrak{o}(\tilde{p},\tilde{q})=\tilde{\mathfrak{t}}\oplus\tilde{\mathfrak{p}},$$ denotes the Cartan decompositions then we have: $$\mathfrak{o}(n)=\mathfrak{t}\oplus i\mathfrak{p}=\tilde{\mathfrak{t}}\oplus i\tilde{\mathfrak{p}}.$$

We shall refer to such a triple: $\Big{(}\mathfrak{o}(p,q),\mathfrak{o}(\tilde{p},\tilde{q}), \mathfrak{o}(n)\Big{)}$, as a \textsl{compatible triple} of real forms. We define this not only for Lie algebras, but also for general real slices of the same dimension:

\begin{defn} Let $W$ and $\widetilde{W}$ be real slices of $(E,g)$. Assume they are both real forms of $W^{\C}\subset (E,g)$. Let $V$ be another real slice of $E$, and a real form of $W^{\C}$, with Euclidean signature. Suppose $W,\widetilde{W}$ and $V$ are pairwise compatible, then a triple: $(W,\widetilde{W}, V)$, will be called a \emph{compatible triple}.  \end{defn}

Note that a compatible triple $(W,\widetilde{W}, V)$, implies that we may choose Cartan involutions of $W,\widetilde{W}$ and $V$ which commute, this is by construction. We shall often refer to $V$ as a \emph{compact real slice} of $E$. In the case of $W\cap\tilde{W}=0$, we note that $W=i\widetilde{W}$, so this corresponds to an anti-isometry, i.e., changing the metric from: $g\mapsto -g$, a standard example is the compatible triple: $\Big{(}i\R\oplus \mathbb{R}, \R\oplus i\R, \R^2\Big{)}$, in $(\C^2, g)$, with $g(-,-)$ the standard holomorphic inner product. 
However there exist compatible triples not of this form, i.e., with $W\cap\widetilde{W}\neq 0$, and to find such examples, it is sufficient to look at compatible triple of Lie algebras. In fact, we may say something  stronger in the case of a compatible triple of semi-simple Lie algebras.

Indeed we now show that if we have a compatible triple of semi-simple Lie algebras $\left(\mathfrak{g},\tilde{\mathfrak{g}},\mathfrak{u}\right)$ with $\mathfrak{u}$ compact like in the example above, then the compact/non-compact parts of the Cartan decompositions of the real forms must intersect. We denote $\mathfrak{t}$ (respectively $\tilde{\mathfrak{t}}$) for the compact part, and $\mathfrak{p}$ (respectively $\tilde{\mathfrak{p}}$) for the non-compact part. This is clear if $\mathfrak{g}=\tilde{\mathfrak{g}}$, so assume they are not equal nor isomorphic.

\begin{prop} \label{i} Assume $\mathfrak{g}\ncong \tilde{\mathfrak{g}}$. We have $\mathfrak{t}\cap\tilde{\mathfrak{t}}\neq0$, and if none of the real forms are compact and they are both simple then also $\mathfrak{p}\cap\tilde{\mathfrak{p}}\neq0$. \end{prop}
\begin{proof} We may assume none of the real forms are compact, because then the first part is trivial. Suppose that $\mathfrak{t}\cap\tilde{\mathfrak{t}}=0$, then it is easy to check that $\mathfrak{t}\subset i\tilde{\mathfrak{p}}$ and $\tilde{\mathfrak{t}}\subset i\mathfrak{p}$. Indeed if $x\in \mathfrak{t}$ then because $\mathfrak{g}\cap\tilde{\mathfrak{g}}=\mathfrak{p}\cap\tilde{\mathfrak{p}}$ then $x=p+i(\tilde{t}+\tilde{p})$ for suitable $p, \tilde{p}\in \tilde{\mathfrak{p}}$ and $\tilde{t}\in \tilde{\mathfrak{t}}$. But then, $$x-i\tilde{p}=p+i\tilde{t}\in \mathfrak{u}\cap i\mathfrak{u},$$ and consequently $x=i\tilde{p}$. The case $\tilde{\mathfrak{t}}\subset i\mathfrak{p}$ is similar. But then $[\mathfrak{t},\mathfrak{t}]\subset \tilde{\mathfrak{t}}\cap \mathfrak{t}$, i.e must be zero, and similarly $[\tilde{\mathfrak{t}},\tilde{\mathfrak{t}}]\subset \tilde{\mathfrak{t}}\cap \mathfrak{t}$. So we conclude that $\tilde{\mathfrak{t}}$ and $\mathfrak{t}$ must be abelian. Now the only simple Lie algebra with abelian compact part is $\mathfrak{sl}(2,\mathbb{R})$, i.e it follows that $$\mathfrak{g}\cong\tilde{\mathfrak{g}}\cong \oplus^k_j \mathfrak{sl}(2,\mathbb{R}),$$ for a suitable $k.$ But since we assume $\mathfrak{g}$ and $\tilde{\mathfrak{g}}$ are non-isomorphic, then this is a contradiction. Now for the second statement suppose $\mathfrak{p}\cap\tilde{\mathfrak{p}}=0$. Then one easily checks that $\tilde{\mathfrak{p}}\subset i\mathfrak{t}$, and it is a standard result that $[\tilde{\mathfrak{p}},\tilde{\mathfrak{p}}]=\tilde{\mathfrak{t}}$ and $[\mathfrak{p},\mathfrak{p}]=\mathfrak{t}$ using that the real forms are simple, and so therefore $\tilde{\mathfrak{t}}\subset [i\mathfrak{t},i\mathfrak{t}]\subset \mathfrak{t}$. Of course we similarly must have $\mathfrak{p}\subset i\tilde{\mathfrak{t}}$, so we conclude that $\mathfrak{t}=\tilde{\mathfrak{t}}$. Hence $\mathfrak{g}\cap\tilde{\mathfrak{g}}=\mathfrak{t}=\tilde{\mathfrak{t}}$. However this will require $\tilde{\mathfrak{p}}\subset i\mathfrak{t}=i\tilde{\mathfrak{t}}$, proving that $\tilde{\mathfrak{p}}=0$. Hence $\tilde{\mathfrak{g}}$ is compact, which contradicts our assumptions. The proposition is proved.  
\end{proof}

\section{Holomorphic Riemannian manifolds}

\subsection{Complexification of real manifolds} 
We will now consider the case where we have a real pseudo-Riemannian manifold. The aim is to consider analytic continuations of such via its complexification and it is thus necessary to assume that the manifold is analytic.  We will also assume that the real dimension of the real manifold, and the complex dimension of the complex manifold are equal (unless stated otherwise). 

Let us first start with a few definitions (see \cite{PV}). 
\begin{defn}
Given a complex manifold $M$ with complex Riemannian metric $g$. If a submanifold $N\subset M$ for any point $p\in N$ we have that $T_pN$ is a real slice of $(T_pM,g)$  (in the sense of  Defn. \ref{realslice}), we will call $N$ a real slice of $(M,g)$. 
\end{defn}
This definition implies that the induced metric from $M$ is real valued on $N$. $N$ is therefore a pseudo-Riemannian manifold. This further implies that real slices are totally real manifolds. 

We will also define the notion of Wick-related spaces, Wick-rotated spaces, as well as a standard Wick-rotation. 
\begin{defn}[Wick-related spaces]
Two pseudo-Riemannian manifolds $P$ and $Q$ are said to be \emph{Wick-related} if there exists a holomorphic Riemannian manifold $(M,g)$ such that $P$ and $Q$ are embedded as real slices of $M$. 
\end{defn}
Wick-related spaces was defined in \cite{PV}. However, we also find it useful to define: 
\begin{defn}[Wick-rotation]
	If two Wick-related spaces intersect at a point $p$ in $M$, then we will use the term \emph{Wick-rotation}: the manifold $P$ can be Wick-rotated to the manifold $Q$ (with respect to the point $p$).  
\end{defn}

\begin{defn}[Standard Wick-rotation]
	Let the $P$ and $Q$ be Wick-related spaces having a common point $p$. Then if the tangent spaces $T_pP$ and $T_pQ$ are embedded: $T_pP, T_pQ\hookrightarrow (T_pP)^{\C}\cong (T_pQ)^{\C}\hookrightarrow T_p M$ such that they form a compatible triple, then we say that the spaces $P$ and $Q$ are related through a \emph{standard Wick-rotation}. 
\end{defn}

We note in the case where $P$ and $Q$ are Wick-rotated by a standard Wick-rotation, and $Q$ is a real slice of Euclidean signature (i.e., it is a Riemannian space), then the tangent spaces: $T_pP$ and $T_pQ$, can be embedded into $(T_pP)^{\mb{C}}\cong (T_pQ)^{\mb{C}}$, such that they are compatible real forms. Also in the case where both real slices: $P$ and $Q$, are Wick-rotated of the same signatures, then they are also Wick-rotated by a standard Wick-rotation. Indeed we can identify $T_pP\cong T_pQ$ (as symmetric non-degenerate bilinear spaces), and in this case the real slices will be compatible with each other (since they are equal as sets in $(T_pP)^{\mb{C}}$). Moreover there is a natural compact real slice: $W\subset (T_pP)^{\mb{C}}$, which is compatible with $T_pP$.

The following proposition is immediate by definition of a compatible triple:

\begin{prop} Two Wick-rotated spaces $P$ and $Q$ by a standard Wick-rotation gives rise to Cartan involutions of $T_pP$ and $T_pQ$ which commute. $\square$  \end{prop}

These three definitions are of increasing speciality; Wick-related spaces need not intersect at a point $p$; nor is there a guarantee that Wick-rotated spaces have commuting Cartan involutions. This all depends on the way the real forms are imbedded into the complexification $O(n,\C)$.

However, in physics, all examples of Wick-rotations (known to the authors) are standard Wick-rotations in the sense above. 

\subsection{Complex differential geometry}
It is useful to review some of the results from complex differential geometry especially in the holomorphic setting. 

A complex Riemannian manifold is a complex manifold $M$ equipped with a symmetric, $\C$-bilinear, non-degenerate form $g$. A vector field is holomorphic if and only if it has holomorphic component functions with respect to any local complex coordinates. The holomorphic tangent bundle $TM$, can be constructed using the construction $E^{1,0}$ via the complexification of $TM$. Similarly, a tensor field $T$ over the holomorphic tangent bundle is holomorphic if and only if the component functions $T^{\mu_{1}...\mu_{l}}_{\phantom{{\mu_{1}...\mu_{l}}}\nu_1...\nu_k}$ are holomorphic with respect to any local holomorphic coordinates $\{z_1,...,z_n\}$ on $M$. Note also that the sum or tensor multiplication of two holomorphic tensors are holomorphic, so is the contraction of a holomorphic tensor. 

For any complex Riemannian manifold there is a unique Levi-Civita connection $\nabla$ (just as in pseudo-Riemannian case) satisfying:
\beq
1. && [X,Y]=\nabla_XY-\nabla_YX \qquad \text{(torsion-free)}, \\
2. && \nabla_Xg=0,   \qquad \text{(metric compatible)}
\eeq
for all vector fields $X$ and $Y$. 

For a holomorphic metric, the Levi-Civita connection $\nabla$ is also holomorphic (as follows from the Koszul equations), so is the Lie bracket. This implies that the holomorphic Riemann curvature tensor, 
\beq
R(X,Y)Z=\nabla_X\nabla_YZ-\nabla_Y\nabla_XZ-\nabla_{[X,Y]}Z,
\eeq
is also holomorphic. Hence, for a holomorphic metric, the connection, and all the curvature tensors inherit this property: they are all holomorphic. 

We remark that this implies that the standard equations for computing the connection coefficients, Riemann curvature tensors etc. which are known from the pseudo-Riemannian case can be used more or less unaltered for a holomorphic Riemannian manifold. Furthermore, this has profound consequences for us as the complexification of a real pseudo-Riemannian manifold is a holomorphic Riemannian manifold. 
Thus, given a real slice $N\subset M$. Then the curvature tensors of $N$ are uniquely extended to a neighbourhood of $N$ in $M$. 

We also remark that since contractions, tensor products preserve holomorphy, polynomial curvature scalars (as considered in \cite{OP}) are also holomorphic and is uniquely determined by knowing them on a real slice. Consequently, we get the following result: If a pseudo-Riemannian space $N$ is obtained from a pseudo-Riemannian space $P$ by Wick-rotation w.r.t. a point $p$, then their polynomial curvature invariants match at $p$. 

Thus, we note the following important facts ($M$ is the ambient holomorphic complex Riemannian manifold) : 
\begin{enumerate}
\item{} For two Wick-related spaces, all the Riemannian curvature tensors can be obtained from $M$ by restricting to the real slices. 
\item{} If a curvature tensor is identically zero for a pseudo-Riemannian manifold, $N$, then it is identically zero in a neighbourhood of $N$ in $M$. 
\end{enumerate}

\subsection{Real slices from a frame-bundle perspective} \label{FB}
In the frame-bundle formulation of differential geometry, the Riemannian case is a frame-bundle with an $O(n)$ structure group. In general, the pseudo-Riemannian case, has a $O(p,q)$ structure group. As we saw earlier, holomorphic Riemannian geometry has $O(n,\C)$ structure group. The relation between the real slices (with structure group $O(p,q)$) and the holomorphic Riemannian case is related through the complexification of $O(p,q)^\C\cong O(n,\C)$. 

In the Wick-rotated case, at the intersection point the different real slices with structure groups $O(p,q)$ and $O(\tilde{p},\tilde{q})$ will both be embedded in $O(n,\C)$. Indeed if $P$ and $Q$ are Wick-rotated at a point $p\in M$ of the same real dimension, say $n$, then $T_pP$ and $T_pQ$ are real slices of $T_pM$, of say signatures $(p,q)$ and $(\tilde{p},\tilde{q})$ respectively. Now since these are totally real spaces, we have natural embeddings of $$(T_pP)^{\C}\hookrightarrow T_pM, \ and \ (T_pQ)^{\C}\hookrightarrow T_pM.$$ We can restrict the metric $g$ on $T_pM$ to $(T_pP)^{\mb{C}}$ and $(T_pQ)^{\mb{C}}$ so they become holomorphic inner product subspaces of $T_pM$. In particular since we have an isomorphism: $(T_pQ)^{\C}\xrightarrow{\psi} (T_pP)^{\C}$ (as holomorphic inner product spaces), then we have natural embeddings: $$T_pP, T_pQ\hookrightarrow (T_pP)^{\C}\subset T_pM,$$ as real slices of $(T_pP)^{\mb{C}}$, with their signatures: $(p,q)$ and $(\tilde{p},\tilde{q})$, respectively. Moreover when restricting to $(T_pP)^{\C}$ on the holomorphic metric $g$ on $M$, gives another holomorphic inner product, with structure group: $$O(n,\C):=\{(T_pP)^{\C}\xrightarrow{f} (T_pP)^{\C}|g(f(-),f(-))=g(-,-)\}.$$ The pseudo-orthogonal groups: $O(p,q)$ (structure group of $P$) and $O(\tilde{p},\tilde{q})$ (structure group of $Q$), will now be embedded as real forms via $\psi$ into $O(n,\mb{C})$.

\

A tensor $x$ over the point $p\in P\cap Q$ w.r.t $P$ will therefore be considered as a vector $x\in V$ for some appropriate vector space, and similarly a tensor $\tilde{x}$ w.r.t to $Q$ over the same point $p$ will be in another real form $\tilde{V}\subset V^{\mathbb{C}}$. This could be, for example, the Riemann tensor or covariant derivatives of the Riemann tensor, restricted to the point. If the two spaces are Wick-rotated the orbits $Gx$, and $\tilde{G}\tilde{x}$ where $G:=O(p,q)$ and $\tilde{G}:=O(\tilde{p},\tilde{q})$, are embedded into the same complex orbit $G^{\C}x\cong G^\C\tilde{x}$, for $G^{\mathbb{C}}:=O(n,\C)$. Hence, we have the two embeddings: 
\[
\begin{CD}
O(n,{\C})\cdot x@=	O(n,{\C})\cdot\tilde{x} \\
@AAA @AAA \\
O(p,q)\cdot x  @.  O(\tilde{p},\tilde{q})\cdot\tilde{x}
\end{CD}
\]
A necessary condition for the existence of Wick-rotated $x$ and $\tilde{x}$ is therefore the existence of a complex orbit in which both real orbits are embedded. In the case of a standard Wick-rotation we know that the tangent spaces $T_pP$ and $T_pQ$ form a compatible triple with a compact real slice, say, $W$ (a real form of $(T_pP)^{\mb{C}}$ of Euclidean signature w.r.t $g$), when embedded into $(T_pP)^{\C}\subset T_pM$. So w.r.t $W$, there is a compact real form: $O(n)$, also embedded in $O(n,\mb{C})$ (as above). Denote now $\mathfrak{o}(p,q)$, $\mathfrak{o}(\tilde{p},\tilde{q})$ and $\mathfrak{o}(n)$, for the real forms (of Lie algebras) of $O(p,q), O(\tilde{p},\tilde{q})$ and $O(n)$ respectively, embedded into $\mathfrak{o}(n,\mb{C})$ (the Lie algebra of $O(n,\mb{C})$) w.r.t a standard Wick-rotation. Then we have the following observation:

\begin{lem}\label{c} The triple of real forms: $\Big{(} \mathfrak{o}(p,q), \mathfrak{o}(\tilde{p},\tilde{q}), \mathfrak{o}(n)\Big{)}$, embedded into $\mathfrak{o}(n,\C)$ under a standard Wick-rotation is also a compatible triple of Lie algebras.     \end{lem}
\begin{proof} Denote $\sigma$ for the conjugation map of $T_pP$ and $\tilde{\sigma}$ for the conjugation map of $T_pQ$, and let $\tau$ be the conjugation map of $W$. We note that the map: $$\mathfrak{o}(n,\C)\rightarrow\mathfrak{o}(n,\C), \ \ f\mapsto \sigma\circ \bar{f},$$ with $\bar{f}(v_1+iv_2):=f(v_1)-if(v_2), \forall v_1,v_2\in \mathfrak{o}(p,q)$, is a conjugation map with fixed points $\mathfrak{o}(p,q)$, this is easy to check. Similarly by replacing $\sigma$ with $\tilde{\sigma}$ and $\tau$ we get conjugation maps associated to $\mathfrak{o}(\tilde{p},\tilde{q})$ and $\mathfrak{o}(n)$. Note that since the conjugation maps: $\sigma, \tilde{\sigma}$ and $\tau$ all commute, implies that the conjugation maps associated to $\mathfrak{o}(p,q), \mathfrak{o}(\tilde{p},\tilde{q})$ and $\mathfrak{o}(n)$ must also commute. The lemma is proved. \end{proof}

\section{Lie groups}

An important class of examples can be found for Lie groups. Lie groups are analytic manifolds and can be equipped with left-invariant metrics of arbitrary signature (at least locally, ignoring the question of global geodesic completeness).  

\subsection{Complex Lie groups and their real forms}

\begin{defn}\label{real} A real Lie subgroup $G$ of a complex Lie group $G^{\C}$ is said to be a \emph{real form} if $\mathfrak{g}$ is a real form of the Lie algebra of $G^{\C}$, and moreover as a group product we have $G^{\C}=GG^{\C}_0$ where $G^{\C}_0$ is the identity component.  \end{defn}

Given a real Lie group, $G$ (which is an analytic manifold), we can complexify the Lie group using the Lie algebra, $\mathfrak{g}$. Since the identity component of the Lie group is determined by a neighbourhood of the identity, the exponential map -- mapping the Lie algebra onto a neighbourhood of the identity -- enables us to complexify the Lie group. The complexified Lie group $G^{\C}$ is then a Lie group of real dimension twice that of the real one: $\dim_{\R}(G^{\C})=2\dim_{\R}(G)$. 

For two real forms $\mathfrak{g}$ and $\tilde{\mathfrak{g}}$ of a complex Lie algebra $\mathfrak{g}^\C$, we can use the exponential map to create embeddings: 
\beq
\begin{CD}
\mathfrak{g} @>>> \mathfrak{g}^\C @<<< \tilde{\mathfrak{g}} \\
@V{\exp}VV @V{\exp}VV @VV{\exp}V \\
G @>>> G^\C @<<< \tilde{G}
\end{CD}
\eeq
Hence, since the exponential map is a holomorphic map, equipping the group $G^\C$ a holomorphic metric, this can be considered as a holomorphic Riemannian manifold. 

Consider therefore a complex Lie group $G^\C$. How can be find possible real forms of this Lie group? Henceforth, we will consider the semi-simple groups having semi-simple real forms. This case was completely classified by Cartan and the existence of real forms hinges on the existence of a \emph{Cartan involution} $\theta$ which is a map $\theta:\mathfrak{g}\rightarrow\mathfrak{g}$, see Def. \ref{CartanInv} and paragraph below.
Such a Cartan involution is unique up to inner automorphisms. 

Consider two real forms (of same dimension) of some complex Lie group $G^{\C}$. Then these necessarily have the unit element in common since the unit element is unique. Hence these real forms are necessarily Wick-rotated with respect to each other. Since they share the unit element, and their corresponding Lie algebras are tangent spaces over the unit element, it makes sense to compare their Cartan involutions. Assume the real forms have Lie algebras $\mathfrak{g}$ and $\mathfrak{g}'$, with corresponding Cartan involutions $\theta$ and $\theta'$, respectively. These two Cartan involutions do not necessarily commute, i.e., $[\theta,\theta']\neq 0$, but they would commute if the real forms are related through a standard Wick-rotation. 

\

Important examples are:
\begin{enumerate}
\item{} The real forms $O(p,q)$ of $O(n,\C)$, where $n=p+q$. 
\item{} The real forms $SU(p,q)$, and $SL(n,\R)$ of $SL(n,\C)$.
\item{} The real forms $Sp(p,q)$	of $Sp(n,\C)$.
\item{} The real forms $G_2$ (compact) and split-$G_2$ of $G_2^\C$. 
\end{enumerate}
 Explicitly, the group $SU(2)$, which can be parameterised by the product of matrices
 \beq \label{su2} \begin{bmatrix}
 	e^{ix_1} & 0 \\
 	0 & e^{-ix_1}
 \end{bmatrix}
 \begin{bmatrix}
 	\cosh(ix_2) & \sinh(ix_2) \\
 	\sinh(ix_2) & \cosh(ix_2)
 \end{bmatrix}
 \begin{bmatrix}
 	\cos(x_3) & \sin(x_3) \\
 	-\sin(x_3) & \cos(x_3)
 \end{bmatrix}
\eeq
By holomorphic extension, $x_k\mapsto z_k=x_k+iy_k$, and then restricting to the real section $(z_1,z_2,z_3)=(iy_1,iy_2,x_3)$ we obtain the following parameterisation of $SL(2,\R)$:
 \beq\label{sl2}
 \begin{bmatrix}
 	e^{-y_1} & 0 \\
 	0 & e^{y_1}
 \end{bmatrix}
 \begin{bmatrix}
 	\cosh(y_2) & -\sinh(y_2) \\
 	-\sinh(y_2) & \cosh(y_2)
 \end{bmatrix}
 \begin{bmatrix}
 	\cos(x_3) & \sin(x_3) \\
 	-\sin(x_3) & \cos(x_3)
 \end{bmatrix}.
\eeq
In this example, the groups $SL(2,\R)$ and $SU(2)$ are related through a standard Wick-rotation. 

\subsection{Example: Split $G_2$-holonomy manifolds}

As an example, let us construct a pseudo-Riemannian split-$G_2$-holonomy manifold from a known Riemannian $G_2$-holonomy manifold \cite{GPP}. 

Let us assume that the metrics is of dimension 7 with $S^3\times S^3$ hypersurfaces. Let us first see how we can use the Killing form to construct an Einstein metric. Since $S^3\cong SU(2)$, we can equip these hypersurfaces with a left-invariant Riemannian metric proportional to the Killing form, $\kappa$, on each factor, i.e., using the left invariant frame, $h(X,Y)=-\lambda \kappa_{S^3\times S^3}(X,Y)$, $\lambda>0$. Since the Killing form is non-degenerate and negative-definite on a compact semi-simple group, the metric $h$ is positive definite. We can now do a Wick-rotation of these Lie groups as explained above to another real form of $SU(2)^\C\cong SL(2,\C)$. The other real form of this complex group is $SL(2,\R)$. The Killing form will like-wise be Wick-rotated to the corresponding form of $SL(2,\R)$. The Killing form of $SL(2,\R)$ is also non-degenerate but of signature $(-++)$. The Wick-rotated form of $h$ will therefore be:
\[\tilde{h}(X,Y)=-\lambda \kappa_{SL(2,\R)\times SL(2,\R)}(X,Y), \] 
which is also non-degenerate but of signature $(----++)$ or (4,2). This would be the standard bi-invariant Einstein metric on $SL(2,\R)\times SL(2,\R)$. 

With some amendments we can also provide with a Wick-rotation of the metrics in \cite{GPP} to a 7-dimensional pseudo-Riemannian manifold of signature $(4,3)$. Let $\sigma^i$ and $\Sigma^i$ be  a set of left-invariant one-forms on two copies of $SU(2)$ so that $\d\sigma^1=-\sigma^2\wedge\sigma^3$  and $\d\Sigma^1=-\Sigma^2\wedge\Sigma^3$ etc. (cyclic permutation). For simplicity, we write $\mbold{\sigma}$ and $\mbold{\Sigma}$ for the column vectors of one-forms:
\[ \mbold{\sigma} =\begin{bmatrix} \sigma^1 \\ \sigma^2 \\ \sigma^3 \end{bmatrix},\qquad  \mbold{\Sigma}= \begin{bmatrix} \Sigma^1 \\ \Sigma^2 \\ \Sigma^3 \end{bmatrix}. \] 
Using the parameterisation eq.(\ref{su2}), the (real) left-invariant one-forms can be written: 
\beq \sigma^1&=&2\left[\cos(2x_3)\d x_2+\cos(2x_2)\sin(2x_3)\d x_1\right]\nonumber \\
\sigma^2&=&2\left[-\sin(2x_3)\d x_2 +\cos(2x_2)\cos(2x_3)\d x_1\right]\nonumber \\
\sigma^3&=&2\left[\d x_3-\sin(2x_2)\d x_1\right],\nonumber
\eeq
 and similar for ${\mbold\Sigma}$.
A $G_2$-holonomy metric can now be written \cite{GPP}: 
\beq
ds^2=\alpha^2\d r^2+\beta^2\left(\mbold\sigma-\mbold A\right)^T\left(\mbold\sigma-\mbold A\right)+\gamma^2{\mbold\Sigma}^T{\mbold\Sigma},
\eeq
where $\alpha=\alpha(r)$, $\beta=\beta(r)$, $\gamma=\gamma(r)$, and the ${\mbold A}$ is the connection one-form ${\mbold A}=\tfrac 12{\mbold\Sigma}$. When
\[ \alpha^2=(1-r^{-3} )^{-1}, ~\beta^2 = \frac 19 r^2 (1-r^{- 3} ), ~\gamma^2=\frac{1}{12}r^2,\] 
this metric is Ricci-flat and has $G_2$-holonomy. 

A Wick-rotation can now be accomplished by Wick-rotating each copy of $S^3$ to $SL(2,\R)$ given explicitly by eqs.(\ref{su2}-\ref{sl2}). By aligning the Wick-rotation for each copy, so that $\mbold{\sigma}$ and $\mbold{\Sigma}$ transform identically, then the metric above can be Wick-rotated into: 
\beq
ds^2=\alpha^2\d r^2+\beta^2\left(\tilde{\mbold\sigma}-\tilde{\mbold A}\right)^T{\mbold \eta}\left(\tilde{\mbold\sigma}-\tilde{\mbold A}\right)+\gamma^2\tilde{\mbold\Sigma}^T{\mbold \eta}\tilde{\mbold\Sigma},
\eeq
where $\tilde{\mbold\sigma}$ and $\tilde{\mbold\Sigma}$ are the corresponding real left-invariant one-forms on two copies of $SL(2,\R)$, $\tilde{\mbold A}=\frac 12\tilde{\mbold\Sigma}$ and ${\mbold \eta}=\text{diag}(-1,-1,1)$. It is clearly essential here that the one-form $(\tilde{\mbold\sigma}-\tilde{\mbold A)}$ is real on $SL(2,\R)\times SL(2,\R)$ but by aligning the Wick-rotations on each copy of $SU(2)$ this can be accomplished. The Wick-rotated metric is therefore Ricci-flat and of signature $(4,3)$. 
As the holonomy group is generated by the Riemann curvature tensor (and its covariant derivatives), the holonomy group of the Wick-rotated space would be another real form of the complexified $G_2^{\C}$ Lie group. Since the only other real form is the split-$G_2$, then the resulting space is of split-$G_2$ holonomy.

\section{A standard Wick-rotation to a real Riemannian space} \label{riemannian}

In Section \ref{FB} we saw a necessarily condition put on the orbits $O(p,q)x$ and $O(\tilde{p},\tilde{q})\tilde{x}$, for the existence of two Wick-rotated spaces $P$ and $Q$ of $M$. We now give a stronger necessary condition on the orbits in the case of a standard Wick rotation to a real Riemannian space. In the Riemannian case we have $\tilde{p}=n$ and $\tilde{q}=0$, so one of the structure groups is a compact real form: $O(n)$, of $O(n,\C)$. 

We will use tools from real GIT applied to semi-simple groups to obtain a necessary condition on the orbits.

\subsection{Minimal vectors and closure of real semi-simple orbits} 

Throughout this section, our groups will always be semi-simple of finitely many components.

\

Let now $G$ be a real semi-simple Lie group with finitely many components, and assume it is a real form of a complex Lie group $G^{\mathbb{C}}$. This immediately implies that our groups are all linear. Suppose $G\xrightarrow{\rho} GL(V)$ is a representation of $G$. We shall say that a complex representation $G^{\mathbb{C}}\xrightarrow{\rho^{\mathbb{C}}} GL(V^{\mathbb{C}})$ is a \textsl{complexified action} of $\rho$ if the following diagram commutes:

\beq
\begin{CD}
G^{\mathbb{C}} @>\rho^{\mathbb{C}}>> GL(V^{\mathbb{C}}) \\
@A{i}AA @A{i}AA \\
G @>\rho>> GL(V)
\end{CD}
\eeq

Let now $\mathfrak{g}\xrightarrow{\theta}\mathfrak{g}$ be a Cartan involution of $\mathfrak{g}$ and $G\xrightarrow{\Theta} G$ be the corresponding unique Cartan involution of $G$ which lifts $\theta$, i.e $d\Theta=\theta$. Denote the Cartan decomposition of $\mathfrak{g}=\mathfrak{t}\oplus \mathfrak{p}$, and similarly let $K\subset G$ be the maximally compact subgroup with Lie algebra $\mathfrak{t}$, so we have Cartan decomposition of $G=Ke^{\mathfrak{p}}$. Let $\mathfrak{u}:=\mathfrak{t}\oplus i\mathfrak{p}$, then $\mathfrak{u}$ is a compact real form of $\mathfrak{g}^{\mathbb{C}}$ which is compatible with $\mathfrak{g}$. As a real Lie group denote the Cartan decomposition $G^{\mathbb{C}}=Ue^{i\mathfrak{u}}$, where $U$ is a compact real form of $G^{\mathbb{C}}$ with Lie algebra $\mathfrak{u}$. Note that $K\subset U$, in this way.

We can choose a $K$-invariant inner product $\langle-,-\rangle$ on $V$ w.r.t the action $\rho$, which is compatible with a $U$-invariant Hermitian inner product $\langle-,-\rangle_{\mathbb{C}}$ on $V^{\mathbb{C}}$ w.r.t $\rho^{\mathbb{C}}$. The inner product $\langle-,-\rangle$ can be chosen such that it has the two properties: 

\begin{enumerate}	
\item{} $d\rho(\mathfrak{t})$ consists of skew-symmetric maps w.r.t $\langle-,-\rangle$.
\item{} $d\rho(\mathfrak{p})$ consists of symmetric maps w.r.t $\langle-,-\rangle$. 

In particular the real part of $\langle-,-\rangle_{\C}$ will have property $(1)$ and $(2)$ with respect to the Cartan decomposition: $(\mathfrak{g}^{\C})_{\R}=\mathfrak{u}\oplus i\mathfrak{u}$. 

\end{enumerate}   For proofs of these facts we refer to \cite{BC} and \cite{RS}.

\

By \textsl{Borel} and \textsl{Chandra} (see \cite{BC}) and also \cite{RS} we have the following theorem relating the closure of the complex orbit to that of the real orbit for real vectors $v\in V$:

\begin{thm}[BC] \label{BC} The following statements hold:
\begin{enumerate}
\item{} Given a real vector $v\in V$ then the real orbit $Gv$ is closed in $V$ w.r.t the classical topology on $V$ if and only if $G^{\mathbb{C}}v$ is closed in $V^{\mathbb{C}}$.
\item{} Given a real vector $v\in V$ then $G^{\mathbb{C}}v\cap V$ is a finite disjoint union of real $G$-orbits.
\end{enumerate}
 \end{thm} 

Two real orbits: $Gv_1$ and $Gv_2$ in the disjoint union $G^{\mathbb{C}}v\cap V$ are often said to be \textsl{conjugate}. Set $||-||^2:=\langle-,-\rangle$, for the norm on $V$.

\begin{defn} A minimal vector $v\in V$ is one which satisfies $$\Big{(}\forall g\in G\Big{)}\Big{(}||g\cdot v||\geq ||v||\Big{)}.$$ \end{defn}

The set of minimal vectors will be denoted by $\mathcal{M}(G,V)\subset V$. Of course one observes the special case where $G$ is a compact group then all vectors are minimal, i.e $\mathcal{M}(G,V)=V$. As an example consider $G:=SL(2,\R)$, and $G^{\C}:=SL(2,\C)$ with the representations being the adjoint actions. Take Cartan involution: $\mathfrak{sl}(2,\R)\rightarrow \mathfrak{sl}(2,\R)$, given by $x\mapsto -x^t$. Then it is not difficult to show that if $x\in \mathfrak{so}(2)$ we have $$G^{\C}x\cap \mathfrak{sl}(2,\R)=Gx\cup G(-x),$$ and if $x\in \mathfrak{p}$, then $G^{\C}x\cap\mathfrak{sl}(2,\R)=Gx$. This classifies what happens in all cases where the complex orbit: $G^{\mathbb{C}}x$, is closed and $x\in \mathfrak{sl}(2,\R)$ (by the theorem below).

\

Now we also have the following theorem by \textsl{Richardson} and \textsl{Slodowy} in \cite{RS}, which relates the closure of a real orbit to the existence of a minimal vector:

\begin{thm}[RS] \label{RS} The following statements hold:
\begin{enumerate}
\item{} A real orbit $Gv$ is closed if and only if $Gv\cap \mathcal{M}(G,V)\neq\emptyset$.
\item{} If $v$ is a minimal vector then $Gv\cap \mathcal{M}(G,V)=Kv$.
\item{} If $Gv$ is not closed then there exist $p\in \mathfrak{p}$ and an $\alpha\in \overline{Gv}$ such that $e^{tp}\cdot v\rightarrow \alpha\in V$ exist as $t\rightarrow -\infty$, and $G\alpha$ is closed. Moreover $G\alpha$  is the unique closed orbit in the closure $\overline{Gv}\subset V$.
\item{} A vector $v\in V$ is minimal if and only if $\Big{(} \forall x\in \mathfrak{p}\Big{)}\Big{(}\langle x\cdot v,v  \rangle=0\Big{)}$, where $x\cdot v$ is the action $d\rho(x)(v)$.
\end{enumerate}
  \end{thm}

Parts (1), (2) and (4) of the theorem is known as the \emph{Kempf-Ness Theorem}, for which it was first proved for linearly complex algebraic groups.

\begin{rem} The above two theorems are proved in an algebraic geometric setting, for which the group is a real linearly reductive group. To see that one can apply these results to semi-simple Lie groups (not necessarily algebraic), we refer also to a \emph{remark} in \cite{EJ}. We recall the fact that any complex semi-simple Lie group is algebraic, and also any holomorphic representation of $G^{\C}$ is also rational. Theorem \ref{RS} in fact also hold for general reductive Lie groups, see \cite{CR}.   \end{rem}

\subsection{Compatible triples and intersection of real orbits}\label{triple}

Our aim in this section is to explore the connection between compatible triples of semi-simple Lie algebras, and the intersection of real orbits. We prove a theorem, stating that if one of the real forms,  $\tilde{G}$ say, is compact, and compatible with another real form $G$ of $G^{\C}$, then two real orbits can only belong to the same complex orbit if they intersect.

We continue with the notation of the previous subsection.

\

So we now apply the well-known theorems of the previous subsection to the situation of compatible triples, i.e suppose we have another real form of $G^{\mathbb{C}}$, say $\tilde{G}$, with Lie algebra $\tilde{\mathfrak{g}}$. We shall assume that the triple $\Big{(}\mathfrak{g},\tilde{\mathfrak{g}},\mathfrak{u}\Big{)}$ is a compatible triple of real forms. So we can choose a Cartan involution of $\tilde{\mathfrak{g}}$ say $\tilde{\theta}$, which commutes with $\theta$ (the Cartan involution of $\mathfrak{g}$), i.e $[\theta,\tilde{\theta}]=0$. Denote $\tilde{G}=\tilde{K}e^{\tilde{\mathfrak{p}}}$ for the Cartan decomposition of $\tilde{G}$, and similarly we denote the local Cartan decomposition $\tilde{\mathfrak{g}}=\tilde{\mathfrak{t}}\oplus \tilde{\mathfrak{p}}$. So we have $K,\tilde{K}\subset U$. In fact by Proposition \ref{i} we can assume that $K\cap \tilde{K}\neq 1$ have a non-trivial intersection.

Consider $\tilde{G}\xrightarrow{\tilde{\rho}} GL(\tilde{V})$ to be a representation of $\tilde{G}$ also with the same complexification $\rho^{\mathbb{C}}$ as $\rho$, with $\tilde{V}\subset V^{\mathbb{C}}$ another real form. We can put similarly a $U$-invariant Hermitian form on $V^{\C}$ (possibly different from the one compatible with $V$), which is compatible with $\tilde{V}$. So we have another commutative diagram, like the one in the previous subsection for $G$:

\beq
\begin{CD}
G^{\mathbb{C}} @>\rho^{\mathbb{C}}>>GL(V^{\mathbb{C}}) \\
@A{i}AA @A{i}AA \\
\tilde{G} @>\tilde{\rho}>> GL(\tilde{V})
\end{CD}
\eeq

A good example to have in mind for such a situation is the compatible triple $\Big{(}\mathfrak{o}(p,q),\mathfrak{o}(\tilde{p},\tilde{q}), \mathfrak{o}(n)\Big{)}$ described in section \ref{compatible}. The action can be the adjoint action $\rho^{\mathbb{C}}:=Ad$ of the complex orthogonal group $O(n,\C)$, with real forms $O(p,q)$ and $O(\tilde{p},\tilde{q})$ with their adjoint actions: $\rho:=Ad$ and $\tilde{\rho}:=Ad$. Here the inner products associated to the adjoint representations can of course be 
\[-\kappa(-,\theta(-)), \ \ -\kappa(-,\tilde{\theta}(-)),\ \ -\kappa(-,\tau(-)) ~(\text{Hermitian}),\] 
where $\tau$ is the conjugation map of $\mathfrak{o}(n)$. Moreover a minimal vector $x\in \mathfrak{o}(p,q)$ in this setting will just be a vector satisfying $[x,\theta(x)]=0$. 

\

We observe that the example of the adjoint action generalises. Indeed if $V$ and $\tilde{V}$ are also compatible real forms of $V^{\mb{C}}$, then we may assume w.l.o.g, that the minimal vectors of our actions: $\mathcal{M}(G,V)$ of $\rho$ and $\mathcal{M}(\tilde{G},\tilde{V})$ of $\tilde{\rho}$, are both contained in the minimal vectors of the complexified action: $\mathcal{M}(G^{\mb{C}},V^{\mb{C}})$ of $\rho^{\mb{C}}$. We refer to \emph{Appendix} A, for a proof of this fact.

Now we prove our main theorem under the compatibility conditions of the Lie algebras and of the vector spaces the groups act on:

\begin{thm} \label{t} Suppose we have a compatible triple of semi-simple Lie algebras: $\Big{(}\mathfrak{g},\tilde{\mathfrak{g}},\mathfrak{u}\Big{)}$. Assume also that $V$ and $\tilde{V}$ are compatible real forms of $V^{\mb{C}}$. Let $v\in V$ and $\tilde{v}\in \tilde{V}$. Then the following statements hold:
\begin{enumerate}
\item{} Suppose $\tilde{v}\in G^{\mathbb{C}}v$. Then if $v\in \mathcal{M}(G,V)$ and $\tilde{v}\in \mathcal{M}(\tilde{G},\tilde{V})$ we have $Uv=U\tilde{v}$, i.e $Kv,\tilde{K}\tilde{v}\subset Uv.$
\item{} If $\tilde{G}:=U$ is the compact real form compatible with $G$, then $$Gv, \tilde{G}\tilde{v}\subset G^{\mathbb{C}}v \Leftrightarrow Gv\cap\tilde{G}\tilde{v}\neq\emptyset.$$
\end{enumerate}
 \end{thm}
\begin{proof} For case (1) if $v$ and $\tilde{v}$ are minimal vectors in $V$ and $\tilde{V}$ respectively then they are also minimal vectors in $V^{\mathbb{C}}$. So $v,\tilde{v}\in G^{\mathbb{C}}v\cap\mathcal{M}(G^{\mathbb{C}},V^{\mathbb{C}})=Uv$, by Theorem \ref{RS}, and (1) follows. For case (2), we note that since $\tilde{G}:=U$ is the compact real form, then $\tilde{G}\tilde{v}$ is closed, and so since $\tilde{G}\tilde{v}\subset G^{\mathbb{C}}v$ then $G^{\mathbb{C}}v$ is also closed and in particular so is the real orbit $Gv$, by Theorem \ref{BC}. Hence we can choose a minimal vector $v_1\in Gv$ which is also minimal in $G^{\mathbb{C}}v$, by Theorem \ref{RS}, i.e $v_1\in \tilde{G}\tilde{v}$, as $\tilde{G}:=U$, and so proves (2). The theorem is proved.  \end{proof}

Although the theorem guarantees intersection between orbits, there are cases where the orbits intersect in a unique vector. Indeed take $G:=SL(2,\mathbb{R})$ and $\tilde{G}:=SU(2)$, and let the action be the adjoint action. The Lie algebras $\mathfrak{sl}(2,\mathbb{R})$ and $\mathfrak{su}(2)$ are naturally compatible, w.r.t to the standard embedding into $\mathfrak{sl}(2,\mathbb{C})$. It is not difficult to show that whenever $Gx, \tilde{G}\tilde{x}$ belong to the same complex orbit: $G^{\mathbb{C}}x$, for $G^{\mathbb{C}}:=SL(2,\mathbb{C})$, then $Gx\cap \tilde{G}\tilde{x}=\{x\}$.

\begin{rem} We point out that case (2) of Theorem \ref{t}, does not require $V$ and $\tilde{V}$ to be compatible real forms of $V^{\mb{C}}$. \end{rem}

\subsection{The real Riemannian case}

Using Theorem \ref{t} we finally derive a necessary condition for the existence of a standard Wick-rotation to a real Riemannian space, following the notation in Section \ref{FB} we have:

\begin{cor} Suppose $P$ and $Q$ are two Wick-rotated spaces of $M$ by a standard Wick-rotation. Let $Q$ be a real Riemannian space. Suppose $x\in V$ and $\tilde{x}\in\tilde{V}$ are two Wick-rotatable tensors. Then the real orbits $O(p,q)x$ and $O(n)\tilde{x}$ intersect, i.e $O(p,q)x\cap O(n)\tilde{x}\neq\emptyset$. \end{cor}
\begin{proof} By Section \ref{FB} , using Lemma \ref{c} we can apply Theorem \ref{t}, and so the result follows. \end{proof}

For general Wick-rotated spaces $P$ and $Q$ by a standard Wick-rotation, we note that if the complex orbit: $O(n,\C)x=O(n,\C)\tilde{x}$, is closed, then a necessary condition is that for the maximally compact subgroups: $$K:=O(p)\times O(q)\subset O(p,q), \ and \ \tilde{K}:=O(\tilde{p})\times O(\tilde{q})\subset O(\tilde{p}, \tilde{q}),$$ the orbits $K\alpha$ and $\tilde{K}\tilde{\alpha}$ (of minimal vectors) must both be embedded into the same compact orbit: $O(n)\alpha$, i.e we have the following diagram of embeddings: 

\[
\begin{CD}
O(n,{\C})\cdot x@=	O(n,{\C})\cdot\tilde{x} \\
@AAA @AAA \\
O(n)\cdot \alpha@=	O(n)\cdot\tilde{\alpha} \\
@AAA @AAA \\
O(p)\times O(q)\cdot \alpha  @.  O(\tilde{p})\times O(\tilde{q})\cdot\tilde{\alpha}
\end{CD}
\] 

This is all by Theorem \ref{t}, noting that the tensor products: $V$ and $\tilde{V}$, for which the groups act on, are compatible real forms of $V^{\mb{C}}$ under a standard Wick-rotation. In order to generalise the previous Corollary, we need to know more about how these $K$-orbits are embedded, and in the case of same signatures one needs to know how many real $K$-orbits there are in the compact orbit: $O(n)\alpha$, i.e when is $l=1$, in the intersection: $$O(n)\alpha\cap V=\cup^l_j Kx_j?$$

There are examples where $l=1$ and $l\neq 1$ for different $O(p,q)$, such examples can be found within the adjoint action, as we will see in the next section.

\subsection{The adjoint action of the Lorentz groups $O(n-1,1)$}

As an example we consider the adjoint action of the Lorentz group $O(n-1,1)\subset O(n,\C)$. We prove that whenever the complex orbit $O(n,\C)\pmb{x}$ is closed for a real vector $\pmb{x}\in \mathfrak{o}(n-1,1)$, then there is a unique real closed orbit in the complex orbit, i.e $O(n,\C)\pmb{x}\cap \mathfrak{o}(n-1,1)=O(n-1,1)\pmb{x}$. We demonstrate that this is the only pseudo-orthogonal group $O(p,q)$ with this property under the adjoint action.
\

Let $\mathfrak{o}(n-1,1)\xrightarrow{\theta}\mathfrak{o}(n-1,1)$, be the Cartan involution given by: $$\pmb{x}\mapsto Ad(I_{n-1,1})(\pmb{x})=\bar{\pmb{x}}.$$ We use the standard norm on $\mathfrak{o}(n-1,1)$, given by $||\pmb{x}||^2:=\lambda\kappa_{\theta}(\pmb{x},\pmb{x})$, where $\lambda>0$, is chosen such that $||\pmb{x}||^2=tr(\pmb{x}^2).$ Observe that here the global Cartan decomposition of $O(n-1,1)$ is given by: $O(n-1,1)=Ke^{\mathfrak{p}}$, where $K=O(n-1)\times O(1)$, with local Cartan decomposition: $\mathfrak{o}(n-1,1)=\mathfrak{t}\oplus \mathfrak{p}$, with $\mathfrak{t}$ consisting of matrices in block form $\mathfrak{t}\cong \mathfrak{so}(n-1)\oplus \mathfrak{so}(1)$, and $\mathfrak{p}$, consists of matrices of the form: $$\begin{pmatrix} 0_{(n-1)\times (n-1)} & iA_{(n-1)\times 1}  \\ -iA^t_{(n-1)\times 1} & O_{1\times 1} \\ \end{pmatrix}, \ \ A\in \mathbb{R}^{n-1}.$$
\

We will borrow the following two standard results from spectral theory:

\begin{lem}[\emph{Spectral theorem for skew-symmetric matrices}] \label{spectral} If $\pmb{x}\in\mathfrak{so}(n)$ then there exist an orthogonal matrix $g\in O(n)$, such that: $$g\pmb{x}g^{-1}=\begin{bmatrix} \mathfrak{so}(2) &  \\  & \mathfrak{so}(2) \\ & & \mathfrak{so}(2) \\ & & & \ddots \\ & & & & \mathfrak{so}(2) \\ & & & & & 0 \\ & & & & & & 0 \\ & & & & & & & \ddots \\ & & & & & & & & 0 \end{bmatrix},$$ where $\mathfrak{so}(2)$ is the $2\times 2$ matrix of the form $\begin{pmatrix} 0 & x \\ -x & 0 \\ \end{pmatrix}$.   \end{lem}

\begin{cor}\label{linear} Let $\pmb{x}_1,\pmb{x}_2\in \mathfrak{so}(n)$ be skew-symmetric matrices, having identical characteristic polynomials. Then there exist $g\in O(n)$ such that $g\textbf{x}_1g^{-1}=\pmb{x}_2$.   \end{cor}

By applying these results, we can prove the following lemma:

\begin{lem} For any vector $\pmb{x}\in \mathfrak{t}\cup \mathfrak{p}$ we have $$O(n,\mathbb{C})\pmb{x}\cap\mathfrak{o}(n-1,1)=O(n-1,1)\pmb{x}.$$  \end{lem}
\begin{proof} Suppose first $\pmb{x}\in \mathfrak{o}(n-1,1)$ and $\pmb{\tilde{x}}\in \mathfrak{o}(n-1,1)$ belong to the compact real part $\mathfrak{t}=\mathfrak{so}(n-1)\oplus\mathfrak{so}(1)$, and moreover lie in the same $O(n,\mathbb{C})$-orbit. Then we can remove the last row and column of the matrices $\pmb{x}$ and $\pmb{\tilde{x}}$, and they will still have the same characteristic polynomial. Call these $\pmb{y}$ and $\pmb{\tilde{y}}$, then they are in $\mathfrak{so}(n-1)$ and so must lie in the same $O(n-1)$-orbit by Corollary \ref{linear}. Now by extending a matrix in $O(n-1)$ to a matrix in $K$ (in the obvious way) then $\pmb{x}$ and $\pmb{\tilde{x}}$ must lie in the same $O(n-1,1)$-orbit.

Consider now the case where $\pmb{x},\pmb{\tilde{x}}\in \mathfrak{p}$ are contained in the non-compact part. It is easy to see by calculating the characteristic polynomials that if $\pmb{x},\pmb{\tilde{x}}$ lie in the same $O(n,\mathbb{C})$-orbit then they lie in the $(n-2)$-sphere: $$S^{n-2}:=\{\pmb{x}\in \mathfrak{p}\Big{|}||\pmb{x}||=1\},$$ where the norm $||,||$ is proportional to the Killing form: $\kappa(\cdot,\cdot)$ on $\mathfrak{o}(n-1,1)$ restricted to $\mathfrak{p}$. Now we already know that $$Ad(K_0)_{|_{\mathfrak{p}}}:=\{\mathfrak{p}\xrightarrow{Ad(k)} \mathfrak{p}|k\in K_0\},$$ where $K_0=SO(n-1)\times SO(1)\cong SO(n-1)$, is contained in the space of isometries (as a closed matrix subgroup): $$Ad(K_0)_{|_{\mathfrak{p}}}\subset Isom(S^{n-2})\cong O(n-1),$$ w.r.t the restricted norm metric. Now observe that $Ad(K_0)_{|_{\mathfrak{p}}}$ has Lie algebra $\cong \mathfrak{o}(n-1,1)$, and is connected. This follows because $\mathfrak{t}\xrightarrow{ad} \mathfrak{gl}(\mathfrak{p})$ is faithful, as $\mathfrak{o}(n-1,1)$ is simple for all $n\geq 1$, hence the kernel of the restricted adjoint action: $K_0\xrightarrow{Ad} GL(\mathfrak{p})$ is a discrete subgroup of $K_0$. So clearly: $$Ad(K_0)_{|_{\mathfrak{p}}}=Isom(S^{n-2})_0\cong SO(n-1),$$ however $Isom(S^{n-2})_0$ acts transitively on $S^{n-2}$, and so proves that $\pmb{x},\pmb{\tilde{x}}$ lie in the same $K_0$-orbit, and hence in the same $O(n-1,1)-$orbit. This proves the lemma.
 \end{proof}

Using the previous lemma we can prove our claim, that $O(n,\C)\pmb{x}$ has a unique real Lorentz orbit: $O(n-1,1)\pmb{x}$, when it is closed, and $\pmb{x}\in \mathfrak{o}(n-1,1)$.

\begin{thm} For any minimal vector $\pmb{x}\in \mathfrak{o}(n-1,1)$ we have $$O(n,\mathbb{C})\pmb{x}\cap\mathfrak{o}(n-1,1)=O(n-1,1)\pmb{x}.$$  \end{thm}
\begin{proof} Let $\pmb{x}$ be a minimal vector of $\mathfrak{o}(n-1,1)$, i.e we can write $$\pmb{x}:=\begin{pmatrix} A \\  & & 0 \\ \end{pmatrix}+\begin{pmatrix} 0_{(n-1)\times (n-1)} & ix  \\ -ix^t & 0 \\ \end{pmatrix},$$ where $A\in \mathfrak{so}(n-1)$ with $Ax=0$. Suppose now that there is another minimal vector: $$\tilde{\pmb{x}}:=\begin{pmatrix} \tilde{A} \\  & & 0 \\ \end{pmatrix}+\begin{pmatrix} 0_{(n-1)\times (n-1)} & i\tilde{x}  \\ -i\tilde{x}^t & 0 \\ \end{pmatrix},$$ belonging to the same complex orbit as $\pmb{x}$. Denote $\pmb{x}=t+p$ and $\tilde{\pmb{x}}=\tilde{t}+\tilde{p}$ for the components defined previously. We may assume $A,\tilde{A}, x, \tilde{x}$ are all non-zero. Let $V_0$ (respectively $\tilde{V}_0$) denote the kernel of the linear maps: $\mathbb{R}^{n-1}\rightarrow \mathbb{R}^{n-1}$, corresponding to the matrices $A$ (respectively $\tilde{A}$). Note that $x\in V_0$ and $\tilde{x}\in \tilde{V}_0$. By the previous lemma, we know that there exist $k_1,k_2\in K$, such that $k_1\cdot t=\tilde{t}$ and $k_2\cdot p=\tilde{p}$. This means that $k_1V_0=\tilde{V}_0$. Suppose $Dim(V_0)=1$. Then $k_1x=\lambda\tilde{x}$, for some $\lambda\in\mathbb{R}$. Now using the norm induced by the inner product on $\mathfrak{p}$: $\kappa_{\theta}(-,-)$, we can say that $||p||=||\tilde{p}||$, i.e $x$ and $\tilde{x}$ have the same Euclidean norm in $\mathbb{R}^{n-1}$. This means that $\lambda=1$, and hence $k_1\cdot \pmb{x}=\tilde{\pmb{x}}$.

Now for the general case, assume $Dim(V_0)>1$. We may assume $A$ has the form $A:=\begin{pmatrix} A_1 &  \\  & 0_{l} \\ \end{pmatrix}$, for some $A_1\in \mathfrak{so}(n-l)$, where $l\geq 2$, and $A_1$ has kernel of dimension 1, as an operator: $\mathbb{R}^{n-l}\rightarrow \mathbb{R}^{n-l}$. Indeed we may assume that $A_1$ has the form: $$\begin{bmatrix} \mathfrak{so}(2) &  \\  & \mathfrak{so}(2) \\ & & \mathfrak{so}(2) \\ & & & \ddots \\ & & & & \mathfrak{so}(2) \\ & & & & & 0 \end{bmatrix},$$ by Lemma \ref{spectral} and Corollary \ref{linear}. Moreover if $x=(x_1,x_2,\dots, x_{n-1})^T$, then denote $x_1:=(x_1,x_2,\dots, x_{n-l})^T\in\mathbb{R}^{n-l}$. One can see that $A_1x_1=0$. We can do this also to $\tilde{A}$, and denote similarly the $(n-l)\times (n-l)$ matrix by $\tilde{A}_1$, and the vector by $\tilde{x}_1$, as $k_1t=\tilde{t}$. There exist $k\in O(n-l)\times O(1)$ such that $k\cdot A_1=\tilde{A}_1$ by Corollary \ref{linear}, i.e we may assume $k_1=\begin{pmatrix}k &  \\  & I_l \\ \end{pmatrix}$. So applying the previous argument to this case, we have $kx_1=\tilde{x}_1$, and hence $k_1\cdot \pmb{x}=\tilde{\pmb{x}}$. This proves the theorem.
 \end{proof}

Now for general $O(p,q)$ with $p,q\neq 1$, then there is always a minimal vector $\pmb{x}\in \mathfrak{o}(p,q)$ such that the closed orbit $O(n,\C)\pmb{x}$ has at least two disjoint real orbits: $O(p,q)\pmb{x}_1$ and $O(p,q)\pmb{x}_2$. To see this, note that if $p,q> 1$ then we can choose the block matrix of the form: $$\pmb{x}:=\begin{bmatrix} \mathfrak{so}(2) & 0 \\ 0 & \mathfrak{so}(2)\\ \end{bmatrix} \in \mathfrak{so}(p)\oplus\mathfrak{so}(q),$$ where the $\mathfrak{so}(2)$-blocks have different characteristic polynomials. Consider the same matrix, but with the blocks interchanged. Call this matrix $\pmb{x}_1\in \mathfrak{so}(p)\oplus \mathfrak{so}(q)$. Then these are in the same $O(n)$-orbit by Corollary \ref{linear}, but can not be related in the same $O(p)\times O(q)$-orbit, hence $O(p,q)\pmb{x}$ and $O(p,q)\pmb{x}_1$ are two disjoint orbits in $O(n,\mathbb{C})\pmb{x}$. 

\subsection{Uniqueness of real orbits and the class of complex Lie groups}

An interesting question is: When is there a unique real orbit in the complex orbit, i.e., when does one have $G^{\C}v\cap V=Gv$, for $v\in V$? In this section we give a class of groups for which this holds. Recall that one such class is the compact groups, easily deduced from the results of \textsl{Richardson} and \textsl{Slodowy}, i.e., Theorem \ref{RS} in \cite{RS}. 

We prove that if $G$ (an arbitrary Lie group) has the structure of a complex Lie group, and $V$ has a complex structure $V\xrightarrow{J} V$, such that $G\xrightarrow{\rho} GL(V)$ is a complex representation w.r.t $J$, then we have uniqueness. We of course assume that $G\subset G^{\mathbb{C}}$ is a real form of some complex Lie group, and we have the following commutative diagram (as before):

\beq
\begin{CD} \label{arrow}
G^{\mathbb{C}} @>\rho^{\mathbb{C}}>> GL(V^{\mathbb{C}}) \\
@A{i}AA @A{i}AA \\
G @>\rho>> GL(V)
\end{CD}
\eeq

\begin{thm}\label{complex} Let $G$ be a complex Lie group, and as a real Lie group let it be a real form of some complex Lie group $G^{\C}$. Suppose $V$ has a complex structure $J$ and the representations are as in the commutative diagram (\ref{arrow}), with $\rho$ a complex representation w.r.t $J$. Then if $v\in V$, we have a unique real $G$-orbit in the complex orbit: $G^{\mathbb{C}}v$, i.e $G^{\mathbb{C}}v\cap V=Gv.$ \end{thm}

We shall prove this statement using the structure theory for Lie groups, as a reference for the results we use, we refer to \cite{Neeb}.

\

Consider now the identity component $G_0$ of $G$, which is a real form of the identity component $G^{\mathbb{C}}_0$. Let $\widetilde{G_0}$ and $\widetilde{G^{\mathbb{C}}_0}$ be the universal covering groups. 

We will now discuss some covering theory and universal complexification of Lie groups, for proofs we again refer to: \cite{Neeb}, for example chapter 15.

Now if $\mathfrak{g}$ has a complex structure $J$ then $\mathfrak{g}$ has the structure of a complex Lie algebra of dimension $\frac{1}{2}Dim(\mathfrak{g})$. So $G_0$ is a complex Lie group. In particular we note that $\mathfrak{g}^{\mathbb{C}}\cong \mathfrak{g}\oplus\bar{\mathfrak{g}}$, where $\bar{\mathfrak{g}}$ is the complex structure on $\mathfrak{g}$ obtained from $-J$. This isomorphism takes $$x+iy\mapsto \Big{(}x+J(y),x-J(y)\Big{)}, \ x,y\in \mathfrak{g},$$ and so therefore $\mathfrak{g}$ can be identified with the set $\{(x,x)|x\in \mathfrak{g}\}$ as a real form of $\mathfrak{g}\oplus\mathfrak{g}$. So the universal complexification group of the universal covering $\widetilde{G_0}$, is just the universal covering: $\widetilde{G_0^{\mathbb{C}}}$, and thus must be isomorphic (as Lie groups) to the product: $$\widetilde{G_0^{\mathbb{C}}}\cong \widetilde{G}_0\times \widetilde{G}_0.$$ Here the right component: $\widetilde{G}_0$, of the product is called the \textsl{opposite complex Lie group} of $\widetilde{G}_0$, since it has complex Lie algebra $\bar{\mathfrak{g}}$. The left and right components of the product are not necessarily isomorphic (as Lie groups) unless $\mathfrak{g}$ (as a complex Lie algebra), has the existence of a real form, like for instance if $\mathfrak{g}$ were reductive. The universal complexification map is simply the diagonal embedding: $$g\mapsto (g,g), \ g\in \widetilde{G}_0.$$ Moreover w.r.t to this map $\widetilde{G}_0$ is a real form of $\widetilde{G_0^{\mathbb{C}}}$ identified as the image: $\{(g,g)|g\in \widetilde{G}_0\}\subset\widetilde{G_0^{\mathbb{C}}}.$

\

An example to have in mind is $G:=O(3,1)$, with $G_0\cong SO(3,\C)$, and $\widetilde{G_0}\cong SL(2,\C)$ with $\widetilde{G_0^{\mathbb{C}}}\cong SL(2,\C)\times SL(2,\C)$.

\

We prove the theorem in steps:

\

\emph{Step 1}. \emph{We extend the action to the universal covering groups.}

Let $\widetilde{G}_0\xrightarrow{p} G_0$ be the universal covering map of $G_0$. By the discussion above, we may assume w.l.o.g that we are dealing with the groups: $\widetilde{G_0^{\mathbb{C}}}:=\widetilde{G}_0\times \widetilde{G}_0$ and the real form: $\widetilde{G}_0:=\{(g,g)|g\in \widetilde{G}_0\}\subset \widetilde{G_0^{\mathbb{C}}}$. Set $\tilde{\mathfrak{g}}$ for the Lie algebra of $\widetilde{G}_0$, and let $\widetilde{G_0^{\mathbb{C}}}\xrightarrow{p_{\mathbb{C}}} G_0^{\mathbb{C}},$ be the unique lift of the Lie isomorphism: $$\tilde{\mathfrak{g}}^{\mathbb{C}}\rightarrow \mathfrak{g}^{\mathbb{C}}$$ induced from the following commutative diagram of Lie algebras: 

\begin{displaymath} \xymatrix{ \tilde{\mathfrak{g}} \ar[r]^{p_{*}} \ar[d]_{i} & \mathfrak{g} \ar[r]^{i} & \mathfrak{g}^{\mathbb{C}} \\ \tilde{\mathfrak{g}}^{\mathbb{C}} \ar[urr]^{\exists!}}\end{displaymath}

Explicitly this unique map is given by: $\tilde{x}+i\tilde{y}\mapsto p_{*}(\tilde{x})+ip_{*}(\tilde{y}), \ \tilde{x},\tilde{y}\in \tilde{\mathfrak{g}}.$ 

The map $p_{\mathbb{C}}$ is a universal covering map of $G_0^{\mathbb{C}}$, simply because $G_0^{\mathbb{C}}$ and the cover is connected, so the map is surjective. In particular we can induce an action of the covering groups on $V$ and $V^{\mathbb{C}}$ using the action: $G^{\mathbb{C}}\xrightarrow{\rho} GL(V^{\mathbb{C}})$ restricted to the identity components: $G_0,G_0^{\mathbb{C}}$. The following commutative diagram of Lie groups illustrates the induced action: 

\begin{displaymath} \xymatrix{ \widetilde{G}_0 \ar[r]^{p} \ar[d]_{i} & G_0 \ar[r]^{i} & G_0^{\mathbb{C}} \ar[d]^{\rho} \\ \widetilde{G_0^{\mathbb{C}}} \ar[urr]^{p_{\mathbb{C}}} \ar[rr]^{\rho\circ p_{\mathbb{C}}} & & GL(V^{\mathbb{C}}) }\end{displaymath}

\emph{Step 2}. \emph{We extend an equivalent action of the covering groups to $V\oplus V\cong V^{\C}$.}

Now since $V$ is a complex vector space with a complex structure $J$, we have an isomorphism $V^{\mathbb{C}}\xrightarrow{\phi} V\oplus V$, given by $$v_1+iv_2\mapsto \Big{(}v_1+J(v_2),v_2-J(v_2)\Big{)}.$$ So we may extend our action to $V\oplus V$ using $\phi$, given by the commutative diagram as follows: 

\begin{displaymath} \xymatrix{ \widetilde{G_0^{\mathbb{C}}} \ar[r]^{\rho\circ p_{\mathbb{C}}} & GL(V^{\mathbb{C}}) \ar[r]^{\phi} & GL(V\oplus V) \\ \widetilde{G_0} \ar[r] \ar[u]^{i}  & GL(V) \ar[u]^{i} \ar[r]^{\phi} & GL(V) \ar[u]^{i} }\end{displaymath}

\emph{Step 3}. \emph{There is a unique real orbit in the complex orbit.}

We claim that for $(g,h)\in \widetilde{G_0^{\mathbb{C}}}$, the action is given by the product action: $$(g,h)\cdot (v_1,v_2):=\Big{(}(g,g)\cdot v_1, (h,h)\cdot v_2\Big{)},$$ where $(g,h)\cdot v_j:=(\rho\circ p_{\mathbb{C}})(g,h)(v_j)$. Indeed since $\widetilde{G_0^{\mathbb{C}}}$ is the universal complexification group, and the product action above is clearly holomorphic, since our action $G\rightarrow GL(V)$ is holomorphic, then it is enough to show that the action of the real form $\widetilde{G_0}$ on $V$ is the product action. By definition we have: $$(g,g)\cdot (v,v):=\phi\Big{(}(g,g)\cdot(\phi^{-1}(v,v))\Big{)}=\phi((g,g)\cdot v)=\Big{(}(g,g)\cdot v, (g,g)\cdot v\Big{)}, \forall v\in V.$$ Now clearly we have a unique real orbit in the complex orbit, i.e: $$\widetilde{G_0^{\mathbb{C}}}v\cap V=\widetilde{G_0}v, \forall v\in V.$$

So finally we derive our theorem: 

\begin{proof} [Proof of theorem \ref{complex}] We note that in order to prove the statement for the groups $G$ and $G^{\mathbb{C}}$, then it is enough to prove the statement for the restricted action of the identity components. This is seen as follows. By definition of $G$ being a real form of $G^{\mathbb{C}}$, (see Definition \ref{real}), this means that $G^{\mathbb{C}}=GG_0^{\mathbb{C}}$ as abstract groups. So if we have $v_1,v_2\in V$ belonging to the same $G^{\mathbb{C}}$-orbit then write $gh\cdot v_1=v_2$ for $g\in G$ and $h\in G_0^{\mathbb{C}}$, so clearly $g^{-1}v_2$ and $v_1$ belong to the same $G_0^{\mathbb{C}}$-orbit. Now if $G_0v_1=G_0g^{-1}v_2$, then also $Gv_1=Gv_2$. Finally to prove it for the identity components it is enough to prove the statement for the induced action of the universal covering groups: $\widetilde{G}_0$ and $\widetilde{G^{\mathbb{C}}_0}$, on $V$ and $V^{\mathbb{C}}$ defined as above. But since the statement has already been proven for this case, then the theorem is proved. \end{proof}

We can naturally apply the theorem to the adjoint action of any complex Lie group $G$. For example $O(3,1)$ has the structure of a complex Lie group, and is the only one among the pseudo-orthogonal Lie groups: $O(p,q)$. So we can apply the theorem to for instance the diagram:

\beq
\begin{CD}
O(4,\C) @>Ad>> GL(\mathfrak{o}(4,\C)) \\
@A{i}AA @A{i}AA \\
O(3,1) @>Ad>> GL(\mathfrak{o}(3,1))
\end{CD}
\eeq

\section{Applications to the pseudo-Riemannian setting}
For the Lorentzian spaces it is useful to use the boost-weight decomposition and the corresponding algebraic classification of tensors \cite{Class}. This turns out to give a very crisp result as to which spaces have non-closed orbits and hence cannot be Wick-rotated to a Riemannian space: 
\begin{thm} 
Given a Lorentzian manifold and assume that (any of) the curvature tensors is of proper type II, III, or N. Then it cannot be Wick-rotated to a real Riemannian space.
\end{thm}
\begin{proof}
Considering the real orbits $Gx$ and $\tilde{G}\tilde{x}$, where $G=O(n)$ and $\tilde{G}=O(n-1,1)$ embedded into the same $G^\C x$. Then by Theorem \ref{BC}, we have that the real orbits are (topologically) closed if and only if $G^\C x$ is closed. Since, $O(n)$ is compact, $Gx$ is necessary compact and closed. Hence, $G^\C x$, is closed, implying $\tilde{G}\tilde{x}$ is closed. However, tensors of proper type II, III, and N, do not have closed orbits, see \cite{minimal}.  	
\end{proof}

This result can be generalised to the pseudo-Riemannian case by classifying the types of tensors that give non-closed orbits. For the pseudo-Riemannian case, the Lie algebra $\mathfrak{g}=\mathfrak{o}(p,q)$ can be split into a positive eigenspace and negative eigenspace of the Cartan involution: 
\[ \mathfrak{g}=\mathfrak{t}\oplus\mathfrak{p}.\]
If the orbit $Gx$ is not closed, then by Theorem \ref{RS} there exists an $X\in \mathfrak{p}$ and a $v_0\in \overline{Gx}\setminus Gx\subset V$ so that $\exp(tX)x\rightarrow v_0$ as $t\rightarrow -\infty$. Thus, this implies the existence of a $v_0$ on the boundary of $Gx$ which is not in $Gx$. The orbit $Gx$ is therefore not closed. Note also that the Lie algebra element $X\in \mathfrak{p}$ generates a one-parameter group $B_t:=\left\{\exp(Xt)~:~ t\in \R\right\}\subset G$ manifesting this limit. 

We recall that a tensor $T$ living on a pseudo-Riemannian manifold can be decomposed using the boost-weight decomposition with respect to a "null-frame" \cite{HC}. Let $k=\min(p,q)$ be the real rank of the group $O(p,q)$. Then, in terms of a orthonormal frame: 
\beq
g(X,Y)=-X_1Y_1-...-X_kY_k+X_{k+1}Y_{k+1}+...+X_nY_n.
\label{metric}\eeq
Let $\mathfrak{a}\subset\mathfrak{p}$ be the largest abelian subalgebra of $\mathfrak{p}$. All such will have dimension equal to $k$. For each $\lambda\in \mathfrak{a}^*$ (the dual of $\mathfrak{a}$) we define
\[ \mathfrak{g}_{\lambda}=\left\{x\in\mathfrak{g} : [y,x]=\lambda(y) x\text{ for } y\in\mathfrak{a}\right\}.\]
 The $\lambda$ is called the \emph{restricted root} of $(\mathfrak{g},\mathfrak{a})$ if $\lambda\neq 0$ and $\mathfrak{g}_\lambda \neq 0$ (see, e.g., \cite{Helgason}). Let $\Sigma$ be the set of restricted roots. The Lie algebra $\mathfrak{g}$ can now be decomposed in the restricted root decomposition: 
\[ \mathfrak{g}=\mathfrak{g}_0\oplus \bigoplus_{\lambda\in\Sigma}\mathfrak{g}_\lambda. \] 
Here, $\mathfrak{g}_0=\mathfrak{a}\oplus\mathfrak{m}$ where $\mathfrak{m}$ is the centraliser of $\mathfrak{a}$ in $\mathfrak{t}$. 

Let now $\mathfrak{a}'\subset\mathfrak{a}$ be the set of regular elements:
\[ \mathfrak{a}'=\left\{ x\in \mathfrak{a} : \lambda(x)\neq 0, \forall \lambda\in\Sigma\right\}.\] 
The set $\mathfrak{a}'$ is the complement of hyperplanes, and let $\mathfrak{a}^+$ be one connected component of $\mathfrak{a}'$ (this is called a Weyl chamber). We say that a root, $\lambda\in\Sigma$, is positive if it has only positive values on $\mathfrak{a}^+$, and  \emph{simple} if it cannot be written as a sum of positive roots. If $\left\{ \alpha_1,...,\alpha_k\right\}$ is the set of simple roots, then the set $\mathfrak{a}^+$ can be given by: 
\[ \mathfrak{a}^+=\left\{ x\in \mathfrak{a} : \alpha_{1}(x)>0, ...,\alpha_k(x)>0\right \}.\] 
By normalising, we can find $k$ linearly independent elements $\mathcal{X}^I\in \mathfrak{a}\subset\mathfrak{p}$, $I=1,...,k$, satisfying the following  criteria: 
\begin{enumerate}	
\item{} $[\mathcal{X}^I,\mathcal{X}^J]=0$, and 
\item{} $\alpha_J(\mathcal{X}^I)=\delta^I_J$. 
\end{enumerate} 
Since $\mathcal{X}^I\in\mathfrak{o}(p,q)\subset\mathrm{End}(T_pM)$, and is symmetric with respect to the inner product  $g_\theta(-,-)$ defined in Defn. \ref{CartanInv}, the eigenvalues are real, and we can find simultaneous eigenvectors of $T_pM$. The corresponding eigenvector decomposition w.r.t. the set $\{\mathcal{X}^I\}$ is identical to the so-called \emph{boost-weight decomposition} \cite{Class}. 

By letting the $\mathcal{X}^I$ act tensorally on $V=\bigoplus T_pM$, an eigenvector decomposition of $V$ can also be achieved. Note that the metric $g$, as a symmetric tensor, is a zero-eigenvector of all $\mathcal{X}^I$ due to the fact that $\mathcal{X}^i\in\mathfrak{o}(p,q)$. Hence, the duality map, $\sharp: T_pM\rightarrow T^*_pM$ induced by the metric $v\overset{\sharp}{\mapsto}v^{\sharp}\equiv g(v,-)$ preserves the boost-weight decomposition. 
Thus an arbitrary tensor $T$ can now be decomposed by the eigenvalues with respect to $\mathcal{X}^I$
\[ T=\sum_{{\bf b}\in \Gamma}\left(T\right)_{\bf b},\]
where $\Gamma\subset \mathbb{Z}^k$ is a finite subset of $\mathbb{Z}^k$.

\begin{lem}
Let $x\in \mathfrak{p}$. Then there exists an $\tilde{x}\in Hx$, where $H$ is the stabilizer of the Cartan involution, so that $\tilde{x}=\lambda_1{\mathcal X}^1+...+\lambda_k{\mathcal X}^k$, where ${\mathcal X}^I$ are the elements given above.	
\end{lem}
\begin{proof}
The stabilizer of the Cartan involution, $H$, fulfills $H^{-1}\theta H=\theta$, and is the largest compact subgroup of $O(p,q)$. In particular, $O(p)\times O(q)\subset H$. In terms of the orthonormal frame, we can write $x\in \mathfrak{p}$ in block-form: 
\[x=\begin{bmatrix}
0 & A \\
A^t & 0
\end{bmatrix}, 
 \]
 where $A$ is a $p\times q$ matrix.  
 The action of $O(p)\times O(q)$ on $x$ induces the following action on $A$: $A\mapsto h^{-1}Ag$ where $(h,g)\in O(p)\times O(q)$. By the singular value decomposition, we can always find a $(h,g)\in O(p)\times O(q)$ so that $h^{-1}Ag$ is diagonal. Thus $h^{-1}Ag=\diag(\lambda_1,...,\lambda_k)$ and $\tilde{x}=\lambda_1{\mathcal X}^1+...+\lambda_k{\mathcal X}^k$.
\end{proof}
Using the representative $\tilde{x}$ rather than $x$, we can now give a criteria for when a tensor does not have a closed orbit. Given a non-closed orbit $Gx$. Then there exists a $\tilde{x}=\lambda_1{\mathcal X}^1+...+\lambda_k{\mathcal X}^k\in \mathfrak{p}$ so that for ${\bf b}=(b_1,...,b_k)$ we have when $t\rightarrow \infty$
\[ B_t(T)=\exp(t\tilde{x})(T)=\sum_{{\bf b}\in \Gamma}\exp[t(b_1\lambda_1+...+b_k\lambda_k)(T)_{\bf b}\rightarrow v_0.\]
For this limit to exist we have either:
\begin{enumerate}
	\item {} $(T)_{\bf b}=0$, or
	\item $b_1\lambda_1+...+b_k\lambda_k\leq 0$,
\end{enumerate}
for all ${\bf b}\in \Gamma.$ Tensors for which such a $\tilde{x}\in \mathfrak{p}$ exists are referred to as tensors possessing the $S^G$-property \cite{SHII}. If $v_0$ is not it the $G$-orbit of $T$, then the orbit is not closed. 
\subsection{Pseudo-Riemannian examples} 
\paragraph{4-dimensional Neutral examples: Walker metrics} The Walker metrics allow for an invariant null-plane and provide with examples of metrics that do not allow for a Wick-rotation. Walker \cite{Walker}  showed that the requirement of an  
invariant $2$-dimensional null plane implies that the (Walker) metric can be written in the canonical form:  
\beq 
\mathrm{d}s^2=2\d u(\d v+A\d u+C\d U)+2\d U(\d V+B\d U), 
\eeq 
where $A$, $B$ and $C$ are functions that  may depend on all of the coordinates. By introducing the null-coframe: 
\beq
\{ e^1,e^2,e^3,e^4\}=\{\d u,\d v+A\d u+C\d U,\d U,\d V+B\d U\}
\eeq
We express the Riemann tensor in terms of this frame so that: 
\[ R=R_{ijkl}e^i\wedge e^i\otimes e^k\wedge e^l \]
Then define the boost-weight of a component, $R_{ijkl}$, as the  a pair $(b_1,b_2)$ where 
\[ (b_1,b_2)=(\#(2)-\#(1),  \#(4)-\#(3)), \] 
where $\#(n)$ means the number of indices equal to $n$. We note that the isometry $\phi:~\{ e^1,e^2,e^3,e^4\} {\mapsto} \{ e^3,e^4,e^1,e^2\}$ interchanges the boost components: $(b_1,b_2)\mapsto (b_2,b_1)$. 

For the Walker metrics, one can easily compute the Riemann tensor and one observes that $R_{ijkl}=0$ if $b_1+b_2>0$. Different functional forms of the functions $A$, $B$, and $C$, gives various possibilities for the remaining components. Whether the orbits form closed orbits or not is summarised in the following table. Here we have assumed that the types are "proper", i.e., it is not possible to find other frames so that it is in a simpler category. 

\noindent
\begin{tabular}{|r|c|c|c|c|c|}
\hline 
$R_{ijkl}$ & $b_1+b_2>0$ & $b_1+b_2<0$ & $0<b_1=-b_2$ &$b_1=-b_2<0$ & Closed? \\
\hline
W1 & $0$ & $\neq 0$ & Any &Any & No \\
W2 & $0$ & $ 0$ & $\neq 0$ & $\neq 0$ & Yes \\
W3 & $0$ & $ 0$ & $ 0$ & $\neq 0$ & No \\
W4 & $0$ & $ 0$ & $ 0$ & $ 0$ & Yes \\
\hline
\end{tabular}\\

The generic Walker metric (type W1) is not closed and thus not possible to Wick-rotate to a Riemannian space. Examples of Walker metrics are given in \cite{SHII}. As simple examples in each category ($a$, $b$, $c$ and $d$ are non-zero constants): 
\beq
\text{W1}:&& \mathrm{d}s_1^2=2\d u(\d v+V\d u)+2\d U(\d V+av^4\d U)\nonumber \\
\text{W2}:&& \mathrm{d}s_2^2=2\d u(\d v+(av^2+bV^2)\d u)+2\d U(\d V+(cv^2+dV^2)\d U) \nonumber \\
\text{W3}:&& \mathrm{d}s_3^2=2\d u(\d v+(av^2+bV^2)\d u)+2\d U(\d V+cV^2\d U)\nonumber \\
\text{W4}:&& \mathrm{d}s_4^2=2\d u(\d v+av^2\d u)+2\d U(\d V+bV^2\d U)
\eeq
Of these, the W4 example metric can be Wick-rotated to a Riemannian space, while the W1 and W3 cannot (in general) due to the fact that they do not have closed orbits. However, both the W3 and  W4 examples can be Wick-rotated to  Lorentzian spaces.

\section*{Acknowledgements} 
This work was supported through the Research Council of Norway, Toppforsk
grant no. 250367: \emph{Pseudo-Riemannian Geometry and Polynomial Curvature Invariants:
Classification, Characterisation and Applications.}

\appendix

\section{On compatible Hermitian inner products}

We follow the notation in Subsection \ref{triple}, and prove an extension of the statement \emph{2.9} in \cite{RS}. Recall that a Hermitian inner product $\langle-,-\rangle$ on $V^{\C}$ is said to be \emph{compatible} with $V$, if $\langle-,-\rangle\in \R$ on $V$. 

\begin{prop} Assume we have a compatible triple: $\Big{(}\mathfrak{g},\tilde{\mathfrak{g}},\mathfrak{u}\Big{)}$. If $V$ and $\tilde{V}$ are compatible real forms of $V^{\C}$, then there exist a $U$-invariant Hermitian inner product on $V^{\C}$ which is compatible with $V$ and $\tilde{V}$.  \end{prop}
\begin{proof} Let $G^{\C}\subset GL(n,\C)$ for $n\geq 1$ minimal. Now since $n\leq Dim_{\R}(\mathfrak{g}^{\C})$, then we may as well assume that $G,\tilde{G}\subset G^{\C}\subset GL(\mathfrak{g}^{\C})$. In particular the conjugation maps $\sigma$ and $\tilde{\sigma}$ of $\mathfrak{g}$ and $\tilde{\mathfrak{g}}$ commute, and are naturally contained in $GL(\mathfrak{g}^{\C})$, by assumption. Denote $\mathfrak{g}^{\C}\xrightarrow{J}\mathfrak{g}^{\C}$ for the complex structure. We now follow the \emph{proof of 2.9} in \cite{RS}, and do a slight change.

Set $(G^{\C})^*\subset GL(\mathfrak{g}^{\C})$ for the closed subgroup generated by $G^{\C}, J, \sigma$ and $\tilde{\sigma}$. Similarly define $K^*$ (respectively $\tilde{K}^*$) to be the compact subgroups of $(G^{\C})^*$ generated by $K, J$ and $\sigma$ (respectively $\tilde{K}, J$ and $\tilde{\sigma}$). Now since our Lie algebras form a compatible triple, then we know that $K,\tilde{K}\subset U$. So we also define $U^*$ to be the compact subgroup of $(G^{\C})^*$ generated by $U, J, \sigma$ and $\tilde{\sigma}$. Then clearly $U^*\cap G^{\mathbb{C}}=U$.

Now let $\sigma_1$ and $\tilde{\sigma}_1$ be the conjugation maps of $V$ and $\tilde{V}$ respectively. Also let $V^{\C}\xrightarrow{\tilde{J}} V^{\C}$ be the complex structure. Now we can easily extend the real representation: $G^{\C}\xrightarrow{\rho^{\C}} GL(V^{\C})$, to a real representation: $(G^{\C})^*\xrightarrow{(\rho^{\C})^*} GL(V^{\mathbb{C}})$, by simply defining: $$(\rho^{\C})^*=\rho^{\C}, \ on \ G^{\C}, \ \ (\rho^{\C})^*(\sigma)=\sigma_1, \ (\rho^{\C})^*(\tilde{\sigma})=\tilde{\sigma}_1, \ and \ (\rho^{\C})^*(J)=\tilde{J}.$$ The map is well-defined since $[\sigma_1,\tilde{\sigma}_1]=0$, as $V$ and $\tilde{V}$ are assumed to be compatible. It now follows that there exist a $U$-invariant Hermitian form, which is compatible with $V$ and $\tilde{V}$, as in the \emph{proof of 2.9} in \cite{RS}. The proposition is proved. 
 \end{proof}

An immediate corollary is the following:

\begin{cor} Assume we have a compatible triple: $\Big{(}\mathfrak{g},\tilde{\mathfrak{g}},\mathfrak{u}\Big{)}$. Let $V$ and $\tilde{V}$ be compatible real forms of $V^{\C}$. Then we can assume w.l.o.g that $\mathcal{M}(G,V)$ (the minimal vectors of $\rho$) and $\mathcal{M}(\tilde{G},\tilde{V})$ (the minimal vectors of $\tilde{\rho}$) are both contained in the same set of minimal vectors of the complexified action $\rho^{\C}$, i.e., we have embeddings  

\[
\begin{CD}
\mathcal{M}(G^{\C}, V^{\C}) @<{i}<< \mathcal{M}(\tilde{G},\tilde{V}) \\
@AA{i}A \\
\mathcal{M}(G,V) 
\end{CD}
\]

\end{cor}

\end{document}